\documentclass[preprint,12pt]{elsarticle}

\usepackage{amssymb}
\usepackage{amsthm}
\usepackage{amsfonts}
\usepackage{amsmath}
\usepackage{caption}
\usepackage{lineno} 
\usepackage{stmaryrd}
\usepackage{listings}
\usepackage{listingsutf8}

\newtheorem{theorem}{Theorem}

\newtheorem{corollary}{Corollary}
\newtheorem{definition}{Definition}

\newtheorem{conjecture}{Conjecture}

\journal{arXiv.org} 

\usepackage{xcolor}

\begin{document}

\begin{frontmatter}


\title{Representations of code loops by binary codes}

\author[1]{Rosemary Miguel Pires}
\ead{rosemarypires@id.uff.br}
\author[2]{Alexandre Grishkov}
\ead{shuragri@gmail.com}
\author[3]{Marina Rasskazova}
\ead{marinarasskanova1@gmail.com}

\address[1]{Departamento de Matem\'atica, Instituto de Ci\^encias Exatas, Universidade Federal Fluminense, 27213-145, Volta Redonda - RJ, Brazil}
\address[2]{Departamento de Matem\'atica, Universidade de S\~ao Paulo, Caixa Postal 66281, 05311-970, S\~ao Paulo - SP, Brazil}
\address[3]{Omsk State Technical University, Omsk, Russia}

\begin{abstract}
Code loops are Moufang loops constructed from doubly even binary codes. Then, given a code loop $L$, we ask which doubly even binary code $V$ produces $L$. In this sense, $V$ is called a representation of $L$. In this article we define and determine all minimal and reduced representations of nonassociative code loops of rank $3$ and $4$.
\end{abstract}


\end{frontmatter}


\section{Introduction}

Code loops were introduced by Robert Griess (1986, \cite{2}) from doubly even binary codes as follows. 

Let $K=\mathbb{F}_{2}=\{0,1\}$ be a field with two elements and $K^n$ be an $n$-dimensional vector space over $K$. Let $u=(u_1,\dots,u_n)$ and $v=(v_1,\dots,v_n)$ in $K^n$ and define $|v|=|\{i\; |\;v_i=1\}|$ (the Hamming weight)  and  $|u \cap v|=|\{i\; |\;u_i=v_i=1\}|.$ A \textit{doubly even binary code} is a subspace $V\subseteq{K^{n}}$ such that $|v|\equiv{0\;}(\mbox{mod}{\;4})$.
%

Let $V$ be a doubly even binary code and the function called \textit{factor set} \linebreak $\phi:V\times V\to \{1,-1\}$ defined by: 

\begin{align}
\phi(v,v)&=(-1)^{\frac{|v|}{4}}, \nonumber \\
\phi(v,w)&=(-1)^{\frac{|v{\cap}w|}{2}}\phi(w,v), \nonumber \\
\phi(0,v)&=\phi(v,0)=1,   \nonumber \\
\phi(v+w,u)&=\phi(v,w+u)\phi(v,w)\phi(w,u)(-1)^{|v{\cap}w{\cap}u|}. \nonumber
\end{align}

In order to define a code loop, let $V$ be a doubly even code and $\phi:V\times V\to \{1,-1\}$ be a factor set. Consider the cartesian product $L(V)= \{1,-1\} \times V $ and define a product $`\cdot'$ in $L(V)$ in the following way
\begin{align}
&v.w = \phi(v,w)(v + w), \nonumber\\
&v.(-w) = (-v).w = -(v.w), \nonumber\\
&(-v).(-w) = v.w. \nonumber
\end{align} 

With this product, we obtain a Moufang loop called \textit{code loop}. We say that $L(V)$ has \textit{rank} $m$, if the dimension of the $K$-vector space $V$ is igual to $m$. Note that $v \in L(V)$ and $-v \in L(V)$ means the elements $(1,v)$ and $(-1,v)$, respectively.

It's important to recall that code loops have some useful properties related to the commutator, associator, and square associated with their elements that we will use to get the main results presented in this article. Chein and Goodaire (1990, \cite{4}) proved that  code loops have a unique nonidentity square, a unique nonidentity commutator, and a unique nonidentity associator. In other words, for any $u,v,w \in V$: 
\begin{align}
v^2 &= (-1)^{\frac{|v|}{4}}0,\nonumber \\
\left[u,v\right] &= u^{-1}v^{-1}uv = (-1)^{\frac{|u \cap v|}{2}}0, \nonumber\\
(u,v,w)&=((uv)w)( (u(vw))^{-1})=(-1)^{|u\cap v\cap w|}0. 
\end{align}

So far we've shown how to build code loops from doubly even binary codes, but we would like to investigate to the other side. In other words, given a certain code loop $L$, we would like to determine the doubly even codes $V$ such that $L\simeq L(V)$. 

A {\textit{representation}} of a given code loop $L$ is a doubly even code $V \subseteq {K^{m}}$ such that $L \simeq L(V)$ and the number $m$ is called \textit{degree of the representation}. We notice that there are many different representations for a same code loop.

In this article we determine all minimal and reduced representations of nonassociative code loops of rank $3$ and $4$ (Sections 3 and 4). In order to do this, first, it is necessary to know all these code loops, up to isomorphism. In \cite{AR}, these code loops were classified using the concept of a characteristic vector associated with them. We recall the main theorems of classification of nonassociative code loops of rank $3$ and $4$ in Section 2. 

\section{Nonassociative Code Loops of Rank 3 and 4}\label{section2}

Let $L$ be a code loop with generator set $X=\left\{x_{1},\dots,x_{n}\right\}$ and center $\left\{1,-1\right\}$. Then we define the characteristic vector of $L$, denoted by $\lambda_{X}(L)$ or $\lambda(L)$, by \[\lambda(L)=(\lambda_{1},\dots,\lambda_{n};\lambda_{12},\dots,\lambda_{1n},\dots,\lambda_{(n-1)n};\lambda_{123},\dots,\lambda_{12n},\dots,\lambda_{(n-2)(n-1)n}),\] where $\lambda_{i}, \lambda_{ij},\lambda_{ijk} \in \mathbb{F}_{2}$, $(-1)^{\lambda_{i}}=x_{i}^{2},\; (-1)^{\lambda_{ij}}=[x_{i},x_{j}]\; \mbox{and}\;  (-1)^{\lambda_{ijk}}=(x_{i},x_{j},x_{k})$.

Here, $[x,y]$ denotes the commutator of $x$ and $y$, and $(x,y,z)$ denotes the associator of $x,y$ and $z$. 

Let $L$ be a nonassociative code loop of rank $3$ with generators $a,b,c.$ Then the associator $(a,b,c)$ is equal to $-1$. The characteristic  vector associated to $L$ is given by $\lambda(L) = (\lambda_1,...,\lambda_6)$ (or simply $\lambda(L) = (\lambda_1...\lambda_6)$) where \linebreak $a^{2}=(-1)^{\lambda_1}$, $b^{2}=(-1)^{\lambda_2}$, $c^{2}=(-1)^{\lambda_3}$, $\left[a,b\right]=(-1)^{\lambda_4}$, $\left[a,c\right]=(-1)^{\lambda_5}$ and $\left[b,c\right]=(-1)^{\lambda_6}.$

\begin{theorem}[\cite{AR}]\label{th2.1}
	Consider $C_{1}^{3},...,C_{5}^{3}$ the code loops with the following characteristic vectors:
	\begin{eqnarray*}
		\begin{array}{lllll}
			\lambda(C_{1}^{3})=(111111), &&\lambda(C_{2}^{3})=(000000),&&\lambda(C_{3}^{3})=(000111),\\
			\lambda(C_{4}^{3})=(110000),&&\lambda(C_{5}^{3})=(100000).&&\\
		\end{array}
	\end{eqnarray*}
	
	Then any two loops from the list $\left\{C_{1}^{3},...,C_{5}^{3}\right\}$  are not isomorphic and all nonassociative
	code loop of rank $3$ is isomorphic to one of this list.
\end{theorem}


The omitted seventh coordinate of  all these characteristic vectors is equal to one (nonassociative case).

For the study of the code loops of rank $4$, we assume that $X=\left\{a,b,c,d\right\}$ is a basis of a nonassociative code loop $L$. In this case $L$ has only one nontrivial associator $(a,b,c)=-1$ and $N(L)=~\textsl{F}_{2}d$ (nucleus of $L$). Thus, a characteristic vector of $L$ is given by $\lambda_{X}(L)=(\lambda_{1},\dots,\lambda_{10})$ (or simply $\lambda(L)=(\lambda_{1}\dots\lambda_{10})$), where $$a^{2}=(-1)^{\lambda_1}, b^{2}=~(-1)^{\lambda_2},c^{2}=(-1)^{\lambda_3}, d^{2}=(-1)^{\lambda_4},$$
$$ \left[a,b\right]=(-1)^{\lambda_5}, \left[a,c\right]=~(-1)^{\lambda_6}, \left[a,d\right]=(-1)^{\lambda_7},$$
$$\left[b,c\right]=~(-1)^{\lambda_8}, \left[b,d\right]=(-1)^{\lambda_9},\left[c,d\right]=(-1)^{\lambda_{10}}.$$

\begin{theorem}[\cite{AR}]\label{th2.2}
	Consider $C_{1}^{4},...,C_{16}^{4}$ the code loops with the following characteristic vectors.
	All nonassociative code loop of rank $4$ is isomorphic to one from the list. Moreover, none of those loops are isomorphic to each other.
	\begin{table}[!hhh]
		\centering
		{\begin{tabular}{llll|lllll|l}
				&  &  & $L$ & $\lambda(L)$ &  &  &  & $L$ & $\lambda(L)$ \\
				\cline{4-5}\cline{9-10}
				&  &  & $C^{4}_{1}$ & $(1110110100)$ &  &  &  & $C^{4}_{9}$ & $(0100001000)$ \\
				&  &  & $C^{4}_{2}$ & $(0000000000)$ &  &  &  & $C^{4}_{10}$ & $(0001111000)$ \\
				&  &  & $C^{4}_{3}$ & $(0000110100)$ &  &  &  & $C^{4}_{11}$ & $(0001001000)$ \\
				&  &  & $C^{4}_{4}$ & $(0010100000)$ &  &  &  & $C^{4}_{12}$ & $(0000001100)$ \\
				&  &  & $C^{4}_{5}$ & $(0000010100)$ &  &  &  & $C^{4}_{13}$ & $(0110111100)$ \\
				&  &  & $C^{4}_{6}$ & $(1111110100)$ &  &  &  & $C^{4}_{14}$ & $(0001001100)$ \\
				&  &  & $C^{4}_{7}$ & $(0001000000)$ &  &  &  & $C^{4}_{15}$ & $(1001001100)$ \\
				&  &  & $C^{4}_{8}$ & $(0000001000)$ &  &  &  & $C^{4}_{16}$ & $(0001111100)$ \\
		\end{tabular}}
		\label{tabvc4}
	\end{table}
	
\end{theorem}
\section{Representations of Code Loops}

We proved by Theorems \ref{th3.4} and \ref{theorem3.6} below that, there are representations of nonassociatives code loops of rank $3$ and $4$ such that the degree of each representation is the smallest possible.

\begin{definition}
	A representation $V$ is called \textbf{minimal} if the degree of $V$ is minimal.
\end{definition}

We identify the ${\mathbb F}_{2}$-space ${{\mathbb F}_{2}^{m}}$ as the set of all subsets of $I_{m}=\left\{1,\dots,m\right\}$ and we define a relation of  equivalence $\sim$ on $I_{m}$:
$i \sim j$ if and only if  $\left\{i,j\right\} \cap v = \left\{i,j\right\}$ or $\left\{i,j\right\} \cap v = \emptyset$, for all $ v \in V$.

Note that this definition is equivalent to: $i \sim j$ if and only if $\left\{i,j\right\} \cap ~ v_{k} = \left\{i,j\right\}$ or $\left\{i,j\right\} \cap ~ v_{k} = \emptyset$, $k=1,\dots,s$ and $\left\{v_{1},\dots,v_{s}\right\}$ is a basis of $V$.

We will consider only representations such that, for any equivalence class $X$, we have $|X| < 8.$ We call these representations by \textbf{reduced representations}.


\begin{definition}
	Let $V$ be a $m$-degree representation of a code loop $L$ and  $X_{1},...,X_{r}$ the equivalence classes over $I_m$. We define the \textbf{type} of $V$ as the vector $(|X_{1}|,...,|X_{r}|)$ such that $|X_{1}| \leq |X_{2}| \leq ...\leq |X_{r}|$. When necessary, we can write $(|X_{1}|...|X_{r}|)$.
\end{definition} 

\begin{definition}
	Let $V_{1}$ and $V_{2}$ be doubly even binary codes of ${\mathbb F}_{2}^{m}$. We say that $V_{1}$ and $V_{2}$ are  isomorphic even codes if and only if there is a bijection $\varphi \in S_{m}$ such that $V_{1}^{\varphi}=V_{2}$.
\end{definition}

\subsection{Reduced Representations of the Code Loops of rank $3$}\label{sec3.1}

Let $C_{i}^{3}$, $i=1,\dots,5,$ the nonassociative code loops of rank $3$ as shown in Section \ref{section2}. We will show how to find all the corresponding reduced representations of these code loops. Denote these representations by $V_{i}^{3}$, $i=1,\dots,5$.

Let $V$ be a doubly even binary code and $v \in V$. If $v^{2}=-1$ then $|v|\equiv{4\;}(\mbox{mod}{\;8})$. Otherwise, we have $|v|\equiv{0\;}(\mbox{mod}{\;8}).$

Now let $u,v \in V$. As a consequence of the definition of doubly even binary code, we obtain $|u\cap v|\equiv{0\;}(\mbox{mod}{\;2})$. Then in the case $[u,v]=-1$, we obtain $|u\cap v|\equiv{2\;}(\mbox{mod}{\;4})$ and in the other case, $|u\cap v|\equiv{0\;}(\mbox{mod}{\;4}).$ 

From these observations, we show in Table \ref{relations} the relations between generators of each representation $V_{i}^{3}$. These relations are important to determine $V_{i}^{3}$, $i=1,\cdots,5$.

\begin{table}[ht]
	\begin{center}
		\caption{Relations between generators of $V_{i}^{3}.$ \label{relations}}
	{\begin{tabular}{@{}ccl@{}} \hline
			$i$ & $\lambda(C_{i}^{3})$ & \ \ \ \ \ \ \ Relations   \\ 			 \hline
			$1$ & $(1,1,1,1,1,1)$ & $|v_{1}|\equiv |v_{2}|\equiv |v_{3}|\equiv4\;(\mbox{mod} \; 8)$  \\
			&                 & $|v_{1}\cap v_{2}|\equiv |v_{1}\cap v_{3}|\equiv |v_{2}\cap v_{3}|\equiv2\;(\mbox{mod}\;4)$ \\
			$2$ & $(0,0,0,0,0,0)$ & $|v_{1}|\equiv |v_{2}|\equiv |v_{3}|\equiv0\;(\mbox{mod} \; 8)$  \\
			&                 & $|v_{1}\cap v_{2}|\equiv |v_{1}\cap v_{3}|\equiv |v_{2}\cap v_{3}|\equiv0\;(\mbox{mod}\;4)$ \\
			$3$ & $(0,0,0,1,1,1)$ & $|v_{1}|\equiv |v_{2}|\equiv |v_{3}|\equiv0\;(\mbox{mod} \; 8)$  \\
			&                 & $|v_{1}\cap v_{2}|\equiv |v_{1}\cap v_{3}|\equiv |v_{2}\cap v_{3}|\equiv2\;(\mbox{mod}\;\;4)$\\
			$4$ & $(1,1,0,0,0,0)$ & $|v_{1}|\equiv |v_{2}|\equiv4\;(\mbox{mod}\; 8)$  \\
			&                 & $|v_{3}|\equiv0\;(\mbox{mod}\;8)$\\
			&                 & $|v_{1}\cap v_{2}|\equiv |v_{1}\cap v_{3}|\equiv |v_{2}\cap v_{3}|\equiv0\;(\mbox{mod}\;4)$\\
			$5$ & $(1,0,0,0,0,0)$ & $|v_{1}|\equiv4\;(\mbox{mod}\; 8)$  \\
			&                 & $|v_{2}|\equiv|v_{3}|\equiv0\;(\mbox{mod}\;8)$\\ 
			&                & $|v_{1}\cap v_{2}|\equiv |v_{1}\cap v_{3}|\equiv |v_{2}\cap v_{3}|\equiv0\;(\mbox{mod}\;4)$\\\hline
	\end{tabular}}
	\end{center}
\end{table}

Let $\left\{v_{i_{1}},v_{i_{2}},v_{i_{3}}\right\}$ be a set of generators of a reduced representation  $V_{i}^{3}$. To simplify the notation, we will write $v_1$, $v_2$ and $v_3$ instead of $v_{i_1}$, $v_{i_2}$ and $v_{i_3}$. We also denote by $t_i=|v_i|$, $t_{ij}=|v_i \cap v_j|$ and $|v_{1}\cap v_{2}\cap v_{3}|= t_{123}$. We have $t_{123}\equiv{1\;}(\mbox{mod}{\;2}),$ since $(v_{1},v_{2},v_{3})=-1$. Because this, it does not happen $v_{j}\subset v_{k}$, $j\neq k$ or $v_{j}\cap v_{k} = \emptyset$. If $v_{1}\subset v_{2}$, for example, we would have $t_{123}\equiv t_{13}\equiv {0\;}(\mbox{mod}{\;2}),$ which would be absurd because of what was seen above. If $v_{2}\cap v_{3} = \emptyset$, for example, we would have $t_{123}\equiv t_{1}\equiv{0\;}(\mbox{mod}{\;4})$, which also can not happen.  

Let's see now how to find a reduced representation. First, we suppose that the degree of  $V_{i}^{3}$ is $m$ and we consider $X$ an equivalence class over $I_{m}=\{1,2,\dots,m\}$ such that $|X|<8$. 

Since $i \sim j$ if and only if $\left\{i,j\right\} \cap ~ v_{k} = \left\{i,j\right\}$ or $\left\{i,j\right\} \cap ~ v_{k} = \emptyset$, for $k=1,2,3$, and $\left\{v_{1},v_{2},v_{3}\right\}$ is a basis of $V_{i}^{3}$, then we only have \;$7$\; equivalence classes given by the elements of the  sets $v_{1}\cap v_{2}\cap v_{3}$, $(v_{1}\cap v_{2})\setminus(v_{1}\cap v_{2}\cap v_{3})$, $(v_{1}\cap v_{3})\setminus(v_{1}\cap v_{2}\cap v_{3})$, $(v_{2}\cap v_{3})\setminus(v_{1}\cap v_{2}\cap v_{3})$, $v_{1}\setminus((v_{1}\cap v_{2})\cup(v_{1}\cap v_{3}))$, $v_{2}\setminus((v_{1}\cap v_{2})\cup(v_{2}\cap v_{3}))$ and $v_{3}\setminus((v_{1}\cap v_{3})\cup(v_{2}\cap v_{3}))$, denoted here, respectively, by $X_{123}, X_{12}, X_{13}, X_{23}, X_{1}, X_{2}$ and $X_{3}$.  As each class has a maximum of $7$ elements, then $t_{123}=1,3,5$ or $7$, $t_{ij}=2,4,6,10$ or $14$ and $t_{i}=4,8,12,16,20,24$ or $28$. Then,  $7\leq m \leq 49$.


The cardinality of the sets $X_{123}$, $X_{ij}$ and $X_{i}$ are denoted by $x_{123}$, $x_{ij}$ or $x_{i}$; $i,j,k=1,2,3$, respectively. To determine a reduced representation of a code loop of rank 3, first we choose values for $t_{123}$ ($t_{123}=x_{123}$), $t_{ij}$ and $t_{i}$, $i,j=1,2,3$, according to Table \ref{relations} and, after we look for the solution for the following linear system in the variables $x_{12}$, $x_{13}$, $x_{23}$, $x_{1}$, $x_{2}$ and $x_{3}$:

%

\begin{equation}\label{eq.matricialsytem.rank3}
Av^t=w,
\end{equation}
where 
\begin{center}
	\begin{math}
	A=\left[ \begin{array}{cccccc}
	1\ \,&0\ \,&0\ \,&0\ \,&0\ \,&0\ \ \\
	0\,&1\,&0\,&0\,&0\,&0\ \\
	0\,&0\,&1\,&0\,&0\,&0\ \\
	1\,&1\,&0\,&1\,&0\,&0\ \\
	1\,&0\,&1\,&0\,&1\,&0\ \\
	0\,&1\,&1\,&0\,&0\,&1\ \\

	\end{array}\right] \end{math}
\end{center}
and   $w=(t_{12}-t_{123}, t_{13}-t_{123}, t_{23}-t_{123}, t_{1}-t_{123},t_{2}-t_{123},t_{3}-t_{123})$.

If the solution $v=(x_{12}, x_{13}, x_{23}, x_1, x_2, x_3)$ satisfies the conditions in Table \ref{condition-rank3}, we can determine the degree $m$ and the type of $V_{i}^{3}$. In fact, $m=x_{123}+x_{12}+x_{13}+x_{23}+x_{1}+x_{2}+x_{3}$ and the type is a vector determined putting the coordinates of $(t_{123}, x_{12}, x_{13}, x_{23}, x_1, x_2, x_3)$  in ascending order.
 
\begin{table}[ht]
\begin{center}
		\caption{ Conditions to find reduced representations of the code loops of rank 3. \label{condition-rank3}}
	{\begin{tabular}{@{}lll@{}} \hline
			$x_{12}\leq 7$ & $x_{13}\leq 7$& $x_{23}\leq 7$ \\
			$t_{12}> t_{123}$ & $t_{13}> t_{123}$ & $t_{23}>  t_{123}$\\
			$x_{1}\leq 7$ & $x_{2}\leq 7$ & $x_{3}\leq 7$ \\
			$t_{1}\geq x_{12}+x_{13}+t_{123}$  & $t_{2}\geq x_{12}+x_{23}+t_{123}$  & $t_{3}\geq x_{13}+x_{23}+t_{123}$\\			\hline
	\end{tabular}}
\end{center}
\end{table}

The generators of $V_{i}^{3}$ can be obtained in the following way: 
 {\footnotesize \begin{equation*}
v_{1}=X_{123}\cup X_{12}\cup X_{13} \cup X_{1},\ \ 
v_{2}=X_{123}\cup X_{12}\cup X_{23} \cup X_{2},\ \ 
v_{3}=X_{123}\cup X_{13}\cup X_{23} \cup X_{3},
\end{equation*}
}where 
\vspace{.1in}

{\footnotesize \begin{tabular}{@{}l@{}} 
	$X_{123}=\{1,\cdots,t_{123}\},$   \\ $X_{12}=\{t_{123}+1,\cdots,t_{123}+x_{12}\},$  \\
	$X_{13}=\{t_{123}+x_{12}+1,\cdots,t_{123}+x_{12}+x_{13}\},$  \\	
	$X_{23}=\{t_{123}+x_{12}+x_{13}+1,\cdots,t_{123}+x_{12}+x_{13}+x_{23}\},$ \\			
	$X_{1}=\{t_{123}+x_{12}+x_{13}+x_{23}+1,\cdots,t_{123}+x_{12}+x_{13}+x_{23}+x_{1}\},$  \\			
	$X_{2}=\{t_{123}+x_{12}+x_{13}+x_{23}+x_{1}+1,\cdots,t_{123}+x_{12}+x_{13}+x_{23}+x_{1}+x_{2}\},$  \\			
	$X_{3}=\{t_{123}+x_{12}+x_{13}+x_{1}+x_{2}+1,\cdots,t_{123}+x_{12}+x_{13}+x_{1}+x_{1}+x_{2}+x_{3}\}.$  \\			
	
\end{tabular}}
\vspace{.1in}

We used the system for computational discrete algebra GAP (Group, Algorithm and Programming)  to solve the system (\ref{eq.matricialsytem.rank3}) with the conditions given by Table \ref{condition-rank3}. 


\subsection{Reduced Representations of the Code Loops of Rank 4}

Let $V$ a $m$-degree reduced representation with generators $v_1,v_2,v_3$ and $v_4$. We denote by $t_i=|v_i|$, $t_{ij}=|v_i \cap v_j|$, $t_{ijk} = |v_i \cap v_i \cap v_k|$ and $t_{1234} = |v_1 \cap v_2 \cap v_3 \cap v_4|$, where $i,j,k=1,2,3,4$. Since $t_{123}\equiv{1\;}(\mbox{mod}\;2)$, $t_{ij4}\equiv{0\;}(\mbox{mod}{\;2})$, $t_{ij}\equiv{0\;}(\mbox{mod}{\;2})$ and $t_{i}\equiv{0\;}(\mbox{mod}{\;4})$ and $V$ is reduced, then $1\leq t_{123}\leq 13$, $0\leq t_{ij4}\leq 14$, $2\leq t_{ij}\leq 28$ $(i,j=1,2,3; i\neq j)$, $0\leq t_{i4}\leq 28$ and $4\leq t_{i}\leq 56$.  We note that the equivalence classes over $I_{m}=\{1,2,\cdots,m\}$ are obtained by calculating the following sets (see Table \ref{t3}).

\begin{table}[ht]
	\begin{center}
		\caption{How to obtain equivalence classes over $I_{m}=\{1,2,\cdots,m\}$ \label{t3}}
	{\begin{tabular}{@{}lll@{}} \hline
			$X_{1234}=v_1\cap v_2 \cap v_3 \cap v_4$ & $X_{123}= (v_1\cap v_2 \cap v_3)\setminus v_4$ & $X_{124}=(v_1\cap v_2 \cap v_4)\setminus v_3$ \\  
			$X_{134}=(v_1\cap v_3 \cap v_4)\setminus v_2$ & $X_{234}=(v_2\cap v_3 \cap v_4)\setminus v_1$ & $X_{12}=(v_1\cap v_2)\setminus (v_3 \cup v_4)$\\ 
			$X_{13}=(v_1\cap v_3)\setminus (v_2 \cup v_4)$& $X_{14}=(v_1\cap v_4)\setminus (v_2 \cup v_3)$& $X_{23}=(v_2\cap v_3)\setminus (v_1 \cup v_4)$\\ 
			$X_{24}=(v_2\cap v_4)\setminus (v_1 \cup v_3)$ & $X_{34}=(v_3\cap v_4)\setminus (v_1 \cup v_2)$ & $X_{1}=v_1\setminus (v_2 \cup v_3 \cup v_4)$ \\ 
			$X_{2}=v_2\setminus (v_1\cup v_3 \cup v_4)$ & $X_{3}=v_3\setminus (v_1 \cup v_2 \cup v_4)$ &  $X_{4}=v_4\setminus (v_1 \cup v_2 \cup v_3)$\\ \hline
	\end{tabular}}
	\end{center}
\end{table}	

As done in previous section, we consider the cardinality of this sets denoted by $x_{1234}$, $x_{ijk}$, $x_{ij}$ or $x_{i}$; $i,j,k=1,2,3,4$. Now we want to determine all reduced representations of the code loops of rank 4. For this, choose values to $t_{1234}=x_{1234}$, $t_{ijk}$, $t_{ij}$ and $t_{i}$, with $i,j,k=1,2,3,4$, and solve the following linear system of $14$ equations in the variables $x_{123}$, $x_{124}$, $x_{134}$, $x_{234}$, $x_{12}$, $x_{13}$, $x_{14}$, $x_{23}$, $x_{24}$, $x_{34}$, $x_{1}$, $x_{2}$, $x_{3}$ and $x_{4}$:

\begin{equation}\label{eq.linearsystem1}\left\lbrace 
\begin{array}{l}
t_{1234}+x_{123}=t_{123}\\
t_{1234}+x_{124}=t_{124}\\
t_{1234}+x_{123}+x_{124}+x_{12}=t_{12}\\
t_{1234}+x_{134}=t_{134}\\
t_{1234}+x_{123}+x_{134}+x_{13}=t_{13}\\
t_{1234}+x_{234}=t_{234}\\
t_{1234}+x_{123}+x_{234}+x_{23}=t_{23}\\
t_{1234}+x_{124}+x_{134}+x_{14}=t_{14}\\
t_{1234}+x_{124}+x_{234}+x_{24}=t_{24}\\
t_{1234}+x_{134}+x_{234}+x_{34}=t_{34}\\
t_{1234}+x_{123}+x_{12}+x_{124}+x_{134}+x_{13}+x_{14}+x_1=t_{1}\\
t_{1234}+x_{123}+x_{12}+x_{124}+x_{234}+x_{23}+x_{24}+x_2=t_{2}\\
t_{1234}+x_{123}+x_{13}+x_{134}+x_{234}+x_{23}+x_{34}+x_3=t_{3}\\
t_{1234}+x_{124}+x_{134}+x_{234}+x_{14}+x_{24}+x_{34}+x_4=t_{4}\\	
\end{array}
\right.
\end{equation}

Rewriting this system as a matrix equation, we have

\begin{equation}\label{eq.matricialsytem}
Av^t=w,
\end{equation}
where 
\begin{center}
	\begin{math}
	A=\left[ \begin{array}{cccccccccccccc}
	1\,&0\,&0\,&0\,&0\,&0\,&0\,&0\,&0\,&0\,&0\,&0\,&0\,&0\\
	0&1&0&0&0&0&0&0&0&0&0&0&0&0\\
	
	1&1&0&0&1&0&0&0&0&0&0&0&0&0\\
	0&0&1&0&0&0&0&0&0&0&0&0&0&0\\
	
	1&0&1&0&0&1&0&0&0&0&0&0&0&0\\
	0&0&0&1&0&0&0&0&0&0&0&0&0&0\\
	
	1&0&0&1&0&0&0&1&0&0&0&0&0&0\\
	0&1&1&0&0&0&1&0&0&0&0&0&0&0\\
	
	0&1&0&1&0&0&0&0&1&0&0&0&0&0\\
	0&0&1&1&0&0&0&0&0&1&0&0&0&0\\
	
	1&1&1&0&1&1&1&0&0&0&1&0&0&0\\
	1&1&0&1&1&0&0&1&1&0&0&1&0&0\\
	
	1&0&1&1&0&1&0&1&0&1&0&0&1&0\\
	0&1&1&1&0&0&1&0&1&1&0&0&0&1
	
	\end{array}\right], \end{math}
\end{center}
and  $w=(t_{123}-t_{1234},t_{124}-t_{1234},t_{12}-t_{1234},t_{134}-t_{1234},t_{13}-t_{1234},t_{234}-t_{1234},t_{23}-t_{1234},t_{14}-t_{1234},t_{24}-t_{1234},t_{34}-t_{1234},t_{1}-t_{1234},t_{2}-t_{1234},t_{3}-t_{1234},t_{4}-t_{1234})$.

Furthermore, the reduced representations will be find for all those solutions that satisfies the conditions:   
\vspace{0.1in}

\begin{table}[ht]
\begin{center}
		\caption{Conditions to find reduced representations of the code loops of rank 4. \label{tb4}}
	{\begin{tabular}{@{}lll@{}} \hline
			$t_{1234}\leq t_{123}$ \;& $x_{123}\geq 0$ & $x_{123}\leq7$ \\ 
			$t_{1234}\leq t_{124}$ & $x_{124}\geq 0$ & $x_{124}\leq7$\\
			$t_{1234}\leq t_{134}$ & $x_{134}\geq 0$ & $x_{134}\leq7$\\
			$t_{1234}\leq t_{234}$  & $x_{234}\geq 0$ & $x_{234}\leq7$\\
			$x_{12}\geq 0$ & $x_{12}\leq 7$ & $t_{12}\geq t_{1234}+x_{123}+x_{124}$\\
			$x_{13}\geq 0$ & $x_{13}\leq 7$ & $t_{13}\geq t_{1234}+x_{123}+x_{134}$\\
			$x_{14}\geq 0$ & $x_{14}\leq 7$ & $t_{14}\geq t_{1234}+x_{124}+x_{134}$\\
			$x_{23}\geq 0$ & $x_{23}\leq 7$ &  $t_{23}\geq t_{1234}+x_{123}+x_{234}$\\
			$x_{24}\geq 0$ & $x_{24}\leq 7$ &  $t_{24}\geq t_{1234}+x_{124}+x_{234}$\\
			$x_{34}\geq 0$ & $x_{34}\leq 7$ &  $t_{34}\geq t_{1234}+x_{134}+x_{234}$\\
			$x_{1}\geq 0$ &  $x_{1}\leq 7$ & $t_{1}\geq x_{12}+x_{13}+x_{14}+x_{123}+x_{124}+x_{134}+t_{1234}$\\
			$x_{2}\geq 0$  &  $x_{2}\leq 7$ & $t_{2}\geq x_{12}+x_{23}+x_{24}+x_{123}+x_{124}+x_{234}+t_{1234}$\\
			$x_{3}\geq 0$ & $x_{3}\leq 7$ & $t_{3}\geq x_{13}+x_{34}+x_{23}+x_{123}+x_{134}+x_{234}+t_{1234}$\\
			$x_{4}\geq 0$ &  $x_{4}\leq 7$ & $t_{4}\geq x_{14}+x_{24}+x_{34}+x_{124}+x_{134}+x_{234}+t_{1234}$\\ \hline
	\end{tabular}}

\end{center}\end{table}

\vspace{-0.4cm}

Thus we need of the following steps to find a reduced representation $V$:


\begin{enumerate}
	\item Choose $t=\left[ t_{1234},t_{123}, t_{124}, t_{134}, t_{234}, t_{12}, t_{13}, t_{14}, t_{23}, t_{24}, t_{34}, t_1, t_2, t_3, t_4\right].$
	\item Write $w=(t_{123}-t_{1234},t_{124}-t_{1234},t_{12}-t_{1234},t_{134}-t_{1234},t_{13}-t_{1234},t_{234}-t_{1234},$$t_{23}-t_{1234},t_{14}-t_{1234},t_{24}-t_{1234},t_{34}-t_{1234},t_{1}-t_{1234},t_{2}-t_{1234},t_{3}-t_{1234},t_{4}-t_{1234})$ and verify if there is a solution $v=(x_{123}, x_{124}, x_{134}, x_{234},$ $x_{12}, x_{13}, x_{14}, x_{23}, x_{24}, x_{34}, x_1, x_2, x_3, x_4)$ for the system $Av^{t}=w$ that satisfies the  conditions given by Table \ref{tb4}. If so, continue to next step.
	\item Find the sets $X_{1234},X_{123}$, $X_{124}$, $X_{134}$, $X_{234}$, $X_{12}$, $X_{13}$, $X_{14}$, $X_{23}$, $X_{24}$, $X_{34}$, $X_{1}$, $X_{2}$, $X_{3}$ and $X_{4}$ using the rules defined in Table \ref{tb6}. Note that the non-empty sets are the equivalence classes over $I_{m}$, where $m$ is the degree of $V$.

	\begin{table}[ht]
		\begin{center}
			\caption{Calculation of the equivalence classes over $I_{m}$. \label{tb6}} 
		{\begin{tabular}{|l|l|l|} \hline
				$X_{1234}$&$\left\lbrace \;\right\rbrace$ \;\mbox{se}\; $x_{1234}=0$ &  \\ 
				&$\left\lbrace 1,\dots,x_{1234}\right\rbrace$ \;\mbox{se}\; $x_{1234}\neq 0$&  \\ \hline
				$X_{123}$ & $\left\lbrace x_{1234}+1,\dots,n_{123}  \right\rbrace$ &  $n_{123}=x_{1234}+x_{123}$ \\ \hline
				$X_{124}$ & $\left\lbrace \; \right\rbrace$ \;\mbox{se}\; $x_{124}=0$ & $n_{124}=n_{123}+x_{124}$\\ 
				&$\left\lbrace n_{123}+1, \dots, n_{124} \right\rbrace$ \;\mbox{se}\; $x_{124}\neq 0$&\\ \hline
				$X_{134}$ & $\left\lbrace \; \right\rbrace$ \;\mbox{se}\; $x_{134}=0$ & $n_{134}=n_{124}+x_{134}$\\ 
				&$\left\lbrace n_{124}+1, \dots, n_{134} \right\rbrace$ \;\mbox{se}\; $x_{134}\neq 0$ &\\ \hline
				$X_{12}$ & $\left\lbrace n_{134}+1, \dots, n_{12} \right\rbrace$ & $n_{12}=n_{134}+x_{12}$\\ \hline
				$X_{13}$ & $\left\lbrace n_{12}+1, \dots, n_{13} \right\rbrace$ & $n_{13}=n_{12}+x_{13}$\\ \hline
				$X_{14}$ & $\left\lbrace \; \right\rbrace$ \;\mbox{se}\; $x_{14}=0$ & $n_{14}=n_{13}+x_{14}$\\ 
				& $\left\lbrace n_{13}+1, \dots, n_{14} \right\rbrace$ \;\mbox{se}\; $x_{14}\neq 0$ & \\ \hline
				$X_{1}$& $\left\lbrace n_{14}+1, \dots, n_1 \right\rbrace$ & $n_1 = n_{14}+x_1$\\ \hline
				$X_{234}$ & $\left\lbrace \; \right\rbrace$ \;\mbox{se}\; $x_{234}=0$  & $n_{234}=n_1 + x_{234}$\\ 
				&$\left\lbrace n_{1}+1, \dots, n_{234} \right\rbrace$ \;\mbox{se}\; $x_{234}\neq 0$ & \\ \hline
				$X_{23}$& $\left\lbrace n_{234}+1, \dots, n_{23} \right\rbrace$ & $n_{23}=n_{234} + x_{23}$ \\ \hline
				$X_{24}$ & $\left\lbrace \; \right\rbrace$ \;\mbox{se}\; $x_{24}= 0$ & $n_{24}=n_{23} + x_{24}$\\ 
				&$\left\lbrace n_{23}+1, \dots, n_{24} \right\rbrace$ \;\mbox{se}\; $x_{24}\neq 0$ &\\ \hline
				$X_{2}$&$\left\lbrace n_{24}+1, \dots, n_{2} \right\rbrace$ & $n_{2}=n_{24}+x_{2}$\\ \hline
				$X_{34}$&$\left\lbrace \; \right\rbrace$ \;\mbox{se}\; $x_{34}= 0$ & $n_{34}=n_{2} + x_{34}$ \\ 
				&$\left\lbrace n_{2}+1, \dots, n_{34} \right\rbrace$ \;\mbox{se}\; $x_{34}\neq 0$ & \\ \hline
				$X_{3}$ &$\left\lbrace n_{34}+1, \dots, n_{3} \right\rbrace$ & $n_{3}=n_{34} + x_{3}$\\ \hline
				$X_{4}$ & $\left\lbrace \; \right\rbrace$ \;\mbox{se}\; $x_{4}= 0$ & $n_{4}=n_{3} + x_{4}$\\ 
				&$\left\lbrace n_{3}+1, \dots, n_{4} \right\rbrace$ \;\mbox{se}\; $x_{4}\neq 0$ &\\ \hline
		\end{tabular}}

	\end{center}	\end{table}
	
	\item Find generators $v_1, v_2, v_3$ and $v_4$ of $V$ writing:
	\vspace{-0.1cm}
	
	\begin{eqnarray*}
		v_1 &=& X_{1234}\cup X_{123}\cup X_{124}\cup X_{134}\cup X_{12} \cup X_{13} \cup X_{14} \cup X_1,\\
		v_2 &=& X_{1234}\cup X_{123}\cup X_{124}\cup X_{234}\cup X_{12} \cup X_{23} \cup X_{24} \cup X_2,\\
		v_3 &=& X_{1234}\cup X_{123}\cup X_{134}\cup X_{234}\cup X_{13} \cup X_{23} \cup X_{34} \cup X_3,\\
		v_4 &=& X_{1234}\cup X_{124}\cup X_{134}\cup X_{234}\cup X_{14} \cup X_{24} \cup X_{34} \cup X_4.
	\end{eqnarray*}
	
\end{enumerate}

As an illustration, we will present a reduced representation, denoted by $V_{16}^1$, for the code loop with characteristic vector $(0001111100)$.

If we consider  $t=\left[ 0,1,0,2,0,2,6,2,6,0,4,8,8,16,4\right]$, then we obtain $w=\left(1,0,2,2,6,0,6,2,0,4,6,8,12,4 \right) $  such that  $v=(1,0,2,0,1,3,0,5,0,2,1,1,3,0)$ is the solution of the system (\ref{eq.matricialsytem}). Then the generators of $V_{16}^1$ are given by:
\begin{eqnarray*}
	v_1&=&\left\lbrace 1,2,3,4,5,6,7,8 \right\rbrace , \\
	v_2&=&\left\lbrace 1,4,9,10,11,12,13,14 \right\rbrace,\\
	v_3&=&\left\lbrace 1,2,3,5,6,7,9,10,11,12,13,15,16,17,18,19 \right\rbrace ,\\
	v_4&=&\left\lbrace 2,3,15,16\right\rbrace .
\end{eqnarray*}

Note that the type of $V^{1}_{16}$ is $(111122335)$. Now let's consider another representation of the same code loop above with the same degree of $V^{1}_{16}$ but with different type. Consider the representation, named as $V^{2}_{16}$, generated by the vectors:
\begin{eqnarray*}
	v_1&=&\left\lbrace 1,2,3,4,5,6,7,8 \right\rbrace , \\
	v_2&=&\left\lbrace 1,2,3,4,5,6,9,10,11,12,13,14,15,16,17,18 \right\rbrace,\\
	v_3&=&\left\lbrace 1,2,3,4,5,7,9,10,11,12,13,14,15,16,17,19 \right\rbrace ,\\
	v_4&=&\left\lbrace 1,2,9,10\right\rbrace .
\end{eqnarray*}

In Table \ref{tb7} we present the weights of the generators of each representation and the weights of all intersections of these generators. We denote $t_{1234}$ by $t$.


\begin{table}[hh!]
	\begin{center}
		\caption{Comparation between $V_{16}^{1}$ and $V_{16}^{2}$. \label{tb7}}
	{\begin{tabular}{c|ccccccccccccccc} 
			& $t$ & $t_{123}$ & $t_{124}$ & $t_{134}$ & $t_{234}$ & $t_{12}$ & $t_{13}$ & $t_{14}$ & $t_{23}$ & $t_{24}$ & $t_{34}$ & $t_{1}$ & $t_{2}$ & $t_{3}$ & $t_{4}$\\ \hline
			$V^{1}_{16}$ &0&1&0&2&0&2&6&2&6&0&4&8&8&16&4\\ 
			$V^{2}_{16}$ &2&5&2&2&4&6&6&2&14&4&4&8&16&16&4\\ 
	\end{tabular}}
	\end{center}
\end{table}
The sets of weights of the elements of the representations $V_{16}^{1}$ and $V_{16}^{2}$, are given, respectively, by $ \left\lbrace 0, 4, 8, 8, 8, 8, 8, 8, 12, 12, 12, 12, 12, 12, 12, 16 \right\rbrace $ and $\left\lbrace  0, 4, 4, 8, 8, 8, 8, 8, 12, 12, 12, 12, 12, 12, 16, 16\right\rbrace $. Since, for instance, there is a unique element of weight $4$ in $V_{16}^{1}$ while in $V_{16}^{2}$ there are $2$ elements of order $4$, we see that these representations are not isomorphic. 

\section{Minimal Representations}

For the next theorems, note that $\left\langle v_{1},v_{2},\cdots, v_{n}\right\rangle$ means the vector subspace of space $\mathbb{F}_{2}^{m}$ generated by $v_{1},v_{2},\cdots,v_{n}$.


\begin{theorem}\label{th3.4}
	The code loops $C_{1}^{3},\dots,C_{5}^{3}$ have the following minimal representations $V_{1},\dots,V_{5}$, which are given by\\
	
	\noindent
	$V_{1}=\left\langle (1,2,3,4),(1,2,5,6),(1,3,5,7)\right\rangle,$\\
	$V_{2}=\left\langle (1,2,3,4,5,6,7,8),(1,2,3,4,9,10,11,12),(1,5,6,7,9,10,11,13)\right\rangle,$\\
	$V_{3}=\left\langle (1,2,3,4,5,6,7,8),(1,2,3,4,5,6,9,10),(1,2,3,4,5,7,9,11)\right\rangle,$\\
	$V_{4}=\left\langle (1,2,3,4),(1,2,5-14),(1,3,5,6,7,8,9,10,11,15,16,17)\right\rangle,$\\
	$V_{5}=\left\langle (1-12),(1-8,13,14,15,16),(1,2,3,4,5,9,10,11,13,14,15,17)\right\rangle.$
\end{theorem}

\begin{corollary}
	Each minimal representation of the code loops $C_{1}^{3},\dots,C_{5}^{3}$ has the following types, respectively:
	$$(1111111),(1111333),(1111115),(1111337),(1113335).$$
\end{corollary}

In order to present the Representation Theorem for nonassociative code loops of rank $4$, note that, according to Theorem \ref{th2.2}, we have exactly $16$ code loops of rank $4$, namely, $C_{1}^{4},C_{2}^{4},\dots,C_{16}^{4}.$ For each $C_{i}^{4}$, $i=1,\dots,16,$ we have to find $V_{i} \subseteq {{\bf F}_{2}^{m}}$ doubly even code of minimal degree $m$ such that $V_{i} \cong L(C_{i}^{4}).$

In general, the set $X=\left\{a,b,c,d\right\}$ will denote a set of generators of $C_{i}^{4}$ and  $V_{i}=\left\langle v_{1},v_{2},v_{3},v_{4} \right\rangle $ a minimal representation of $C_{i}^{4}$, where $v_{1},v_{2},v_{3},v_{4}$ are vectors that correspond to $a,b,c,d$ respectively.

For the next theorem we use the notation: $t_{ijk}=|v_{i}\cap v_{j}\cap v_{k}|, i,j,k = 1,..,4$ and $t_{1234}=|v_{1}\cap v_{2}\cap v_{3}\cap v_{4}|.$

\begin{theorem}\label{theorem3.6}
	Each code loop $C_{1}^{4},\dots,C_{16}^{4}$ has the following set of generators to its minimal representation $V_{1},\dots,V_{16}$, respectively:\\
	
	\noindent
	$V_{1}=\left\langle (1,2,3,4),(1,2,5,6),(1,3,5,7),(1-8)\right\rangle,$\\
	$V_{2}=\left\langle (1-8),(1-4,9-12),(1,5,6,7,9,10,11,13),(1,2,5,8,9,12,13,14)\right\rangle,$\\
	$V_{3}=\left\langle (1-8),(1-6,9,10),(1,2,3,4,5,7,9,11),(1,6-12)\right\rangle,$\\
	$V_{4}=\left\langle (1-8),(1-6,9,10),(1,2,3,7,9,11-17),(1,4,7,8,9,10,11,18)\right\rangle,$\\
	$V_{5}=\left\langle (1-8),(1,2,3,4,9,10,11,12),(1,5,9,13-17),(1,2,5,6,9,10,13,18)\right\rangle,$\\
	$V_{6}=\left\langle (1,2,3,4),(1,2,5,6),(1,3,5,7),(8,9,10,11)\right\rangle,$\\
	$V_{7}=\left\langle (1-8),(1,2,3,4,9,10,11,12),(1,5,6,7,9,10,11,13),(14,15,16,17)\right\rangle,$\\
	$V_{8}=\left\langle (1-8),(1,2,3,4,9-12),(1,2,3,5,9,13,14,15),(1,2,10,11,13,14,16,17)\right\rangle,$\\
	$V_{9}=\left\langle (1-8),(1,2,3,4,9-16),(1,5,6,7,9,10,11,17),(5,6,9,10,12,13,18,19)\right\rangle,$\\
	$V_{10}=\left\langle (1-8),(1,2,9-14),(1,3,9,10,11,15,16,17),(4,5,18,19)\right\rangle,$\\
	$V_{11}=\left\langle (1-8),(1,2,3,4,9-12),(1,2,3,5,9,13,14,15),(6,7,16,17)\right\rangle,$\\
	$V_{12}=\left\langle (1-8),(1,2,3,4,9-12),(1,2,3,5,9,10,11,13),(1,2,9,10,14,15,16,17)\right\rangle,$\\
	$V_{13}=\left\langle (1-8),(1,2,9,10),(1,3,9,11),(4,5,12,13,14,15,16,17)\right\rangle,$\\
	$V_{14}=\left\langle (1-8),(1,2,3,4,9-12),(1,2,3,5,9,10,11,13),(1,2,9,10)\right\rangle,$\\
	$V_{15}=\left\langle (1-12),(1,2,3,4,13-16),(1,2,3,5,13,14,15,17),(1,2,13,14)\right\rangle,$\\
	$V_{16}=\left\langle (1-8),(1,2,9,10,11,12,13,14),(1,3,9,10,11,12,13,15),(4,5,16,17)\right\rangle.$
\end{theorem}

\begin{proof}
	Each $V_{i}$, $i=1,\dots,16$, is clearly a doubly even binary code. Now, to prove that $V_{i}\simeq L(C_{i}^{4})$ we just need to find the characteristic vector associated to $L(C_{i}^{4})$ and apply the Theorem \ref{th2.2} (Classification of code loop of rank 4). Therefore, $V_{i}$ is a representation of  $C_{i}^{4}$, $i=1,\dots,16$.
	
	Now, we are going to prove, up to isomorphism, that $V_{1}$ is the unique minimal representation of $C_{1}^{4}$. We consider $X=\{a,b,c,d\}$ a set of generators of $C_{1}^{4}$ such that $\lambda=\lambda(C_{1}^{4})=(1110110100)$. We suppose that $V=\left\langle v_{1},v_{2},v_{3},v_{4}\right\rangle $ is a minimal representation of $C_{1}^{4}$, where $v_{1},v_{2},v_{3},v_{4}$ corresponds to $a,b,c,d$, respectively. Hence, $\mbox{deg} \;V \leq 8$. In this case, we have 
	\begin{eqnarray*}
		|v_{1}|&\equiv & |v_{2}|\equiv |v_{3}|\equiv 4\;(\mbox{mod}\;8)\\
		|v_{4}|&\equiv & 0\;(\mbox{mod}\;8)\\
		|v_{1}\cap v_{2}|&\equiv & |v_{1}\cap v_{3}|\equiv |v_{2}\cap v_{3}|\equiv 2 \;(\mbox{mod}\;4)\\
		|v_{1}\cap v_{4}|&\equiv & |v_{2}\cap v_{4}|\equiv |v_{3}\cap v_{4}|\equiv 0 \;(\mbox{mod}\;4)
	\end{eqnarray*}
	
	Suppose $v_{1}=(1,2,3,4)$, then $|v_{1}\cap v_{2}| = |v_{1}\cap v_{3}|=2$ and hence, $t_{123}=1$. Then we can assume $v_{2}=(1,2,5,6)$ and $v_{3}=(1,3,5,7)$.
	
	We will analyse two possible cases for values of $t_{1234}$: $0$ and $1$.
	
	
	Case $t_{1234}=0$, we have $t_{ij4}=0$ and thus, $|v_{i}\cap v_{4}|=0$.  Therefore, $|v_{4}|\geq 8$ and then, we don't have minimal reduced representation in this case. Case $t_{1234}=1$, we have $t_{ij4}=2$ and thus, $|v_{i}\cap v_{4}|=4$, that is, $v_{i}\subset v_{4}$, $i=1,2,3$. Next, $v_{4}=(1-8)$. Therefore, $V=V_{1}$. If $v_{1}=(1-12)$ we would have $\mbox{deg}\; V > 12$, which contradicts the minimality of $V$.
	
	Now, we will prove that $V_{7}$, up to isomorphism, is the unique minimal representation of $C_{7}^{4}$. Here the characteristic vector is given by $\lambda=(0001000000)$. We suppose that $V=\left\langle v_{1},v_{2},v_{3},v_{4}\right\rangle $ is a minimal representation of $C_{7}^{4}$. Then:
	\begin{eqnarray*}
		|v_{1}|&\equiv & |v_{2}|\equiv |v_{3}|\equiv 0\;(\mbox{mod}\;8)\\
		|v_{4}|&\equiv & 4\;(\mbox{mod}\;8)\\
		|v_{1}\cap v_{2}|&\equiv & |v_{1}\cap v_{3}|\equiv |v_{2}\cap v_{3}|\equiv 0 \;(\mbox{mod}\;4)\\
		|v_{1}\cap v_{4}|&\equiv & |v_{2}\cap v_{4}|\equiv |v_{3}\cap v_{4}|\equiv 0 \;(\mbox{mod}\;4)
	\end{eqnarray*}
	
	Let $v_{1}=(1-8)$, so $|v_{1}\cap v_{2}|=|v_{1}\cap v_{3}|=4$ and hence, $t_{123}=1$ or $3$. Suppose $v_{2}=(1-4,9-12)$, so $|v_{2}\cap v_{3}|=4$. Case $t_{123}=1$, we consider $v_{3}=(1,5-7,9-11,13)$. If $t_{1234}=0$, then $t_{ij4}=0$ or $2$. Considering that $|v_{i}\cap v_{j}|\equiv 0 \;(\mbox{mod}\;4)$, $i,j=1,2,3$, $i\neq j$, then we have only two subcases to analyze:
	\begin{itemize}
		\item $t_{ij4}=0$: In this subcase, $|v_{i}\cap v_{4}|=0$, for $ i=1,2,3$ and hence, we can assume \\$v_{4}=(14,15,16,17)$. Thus, for this case, $V=V_{7}$. 
		\item $t_{ij4}=2$: Here, $|v_{i}\cap v_{4}|=4$, for $ i=1,2,3$ and then,  $v_{4}=(2,3,5,6,9,10,14-19)$, which contradicts the minimality of $V$.
	\end{itemize}
	
	If $t_{1234}=1$, then $t_{ij4}=2$ or $4$. Analogously, we have two subcases to analyze:
	\begin{itemize}
		\item $t_{ij4}=2$: In this subcase, $|v_{i}\cap v_{4}|=4$, for $ \; i=1,2,3$, which produces $\mbox{deg}\; V > 17$, a contradiction.
		\item $t_{ij4}=4$: In this subcase, $|v_{i}\cap v_{4}|=8$, for $ \; i=1,2,3$, which also contradicts the minimality of $V$.
	\end{itemize}
	
	Now, analyzing the case $t_{123}=3$, we can suppose that $v_{3}=(1,2,3,5,9,13-15)$. If $t_{1234}=0,$ then $|v_{i}\cap v_{4}|=0$ and hence, we will have $v_{4}=(16,17,18,19)$, which contradicts the minimality of $V$. Analogously, for the cases $t_{1234}=1,2$ and $4$, we will have $\mbox{deg}\; V > 17$. 
	
	Now, let  $C_{10}^{4}$ be the code loop with $\lambda=(0001111000)$ and $V=\left\langle v_{1},v_{2},v_{3},v_{4}\right\rangle $ its minimal representation. Let $v_{1}=(1-8)$, then $|v_{1}\cap v_{2}|=2$ or $6$ and $|v_{1}\cap v_{3}|=2$ or $6$. Case $|v_{1}\cap v_{2}|=2$, we can assume $v_{2}=(1,2,9-14)$. Hence, we have $t_{123}=1$ and $|v_{2}\cap v_{3}|=4$. 
	\begin{itemize}
		\item For $|v_{1}\cap v_{3}|=2$, consider $v_{3}=(1,3,9-11,15-17)$. If $t_{1234}=0$, then $t_{124}=t_{134}=0$ and $t_{234}=0$ or $2$. Thus $|v_{1}\cap v_{4}|=2$. If $t_{234}=0$: $|v_{2}\cap v_{4}|=0$ and $|v_{3}\cap v_{4}|=0$. Thus, we have $v_{4}=(4,5,18,19)$ and, therefore, $V=V_{10}$. In the case $t_{234}=2$ we will find $v_{4}=(4,5,9,10,12,13,15,16,18-21)$, contradicting the minimality of $V$. The analyze of $t_{1234}=1$ is analogous.
		
		\item For $|v_{1}\cap v_{3}|=6$, we can consider $v_{3}=(1,3-7,9-11,15-21)$. In this case $\mbox{deg}\; V > 19$ for any analysis.
		
	\end{itemize}
	
	We don't have minimal representation in case $|v_{1}\cap v_{2}|=6$. 
	
	Analogously, in other cases, we prove that each $V_{i}$ is the unique minimal representation, up to isomorphism.

\end{proof}


\begin{corollary}Each minimal representation of the code loops $C_{1}^{4},\dots,C_{16}^{4}$ has the following degree and type, respectively:
\end{corollary}

	\begin{center}
				{\begin{tabular}{l|l|l|l|l|l} 
			\multicolumn{1}{c|}{$i$} & \multicolumn{1}{c|}{$\mbox{deg}\;V_{i}$} & \multicolumn{1}{c|}{type of $V_{i}$}&\multicolumn{1}{c|}{$i$} & \multicolumn{1}{c|}{$\mbox{deg}\;V_{i}$} & \multicolumn{1}{c}{type of $V_{i}$} \\ 
			\hline
			\multicolumn{1}{c|}{1} & \multicolumn{1}{c|}{8} & (11111111) & 
			\multicolumn{1}{c|}{2} & \multicolumn{1}{c|}{14} & (11111111222) \\ 
			\multicolumn{1}{c|}{3} & \multicolumn{1}{c|}{12} & (111111114) &
			\multicolumn{1}{c|}{4} & \multicolumn{1}{c|}{18} & (11111111226) \\ 
			\multicolumn{1}{c|}{5} & \multicolumn{1}{c|}{18} & (111111112224) &
			\multicolumn{1}{c|}{6} & \multicolumn{1}{c|}{11} & (11111114) \\ 
			\multicolumn{1}{c|}{7} & \multicolumn{1}{c|}{17} & (11113334) &
			\multicolumn{1}{c|}{8} & \multicolumn{1}{c|}{17} & (11111122223) \\ 
			\multicolumn{1}{c|}{9} & \multicolumn{1}{c|}{19} & (11111222233) &
			\multicolumn{1}{c|}{10} & \multicolumn{1}{c|}{19} & (111223333) \\ 
			\multicolumn{1}{c|}{11} & \multicolumn{1}{c|}{17} & (111122333) & 
			\multicolumn{1}{c|}{12} & \multicolumn{1}{c|}{17} & (1111112234) \\ 
			\multicolumn{1}{c|}{13} & \multicolumn{1}{c|}{17} & (111111236) &
			\multicolumn{1}{c|}{14} & \multicolumn{1}{c|}{13} & (111111223) \\ 
			\multicolumn{1}{c|}{15} & \multicolumn{1}{c|}{17} & (111111227) &
			\multicolumn{1}{c|}{16} & \multicolumn{1}{c|}{17} & (111112235) \\ 
	\end{tabular}}
	\end{center}

Note that in the case of code loops of rank 3 and 4 the type of code loop define this loop up to isomophism. May be it is true in general case.

\begin{conjecture}
	Let $V_{1}$ and $V_{2}$ be representations of a code loop $L$. If this representations have the same degree and type, then $V_{1}$ and $V_{2}$ are isomorphic. 
\end{conjecture}

\section*{Acknowledgments}

The first author thanks  Institute of Mathematics and Statistics, University of S\~ao Paulo, Brazil that gave the space to realize this research. The second author thanks  FAPESP and CNPq for financial support. The third author thanks FAPESP(Brazil), grant 2018/11292-6 for financial support.


\end{document}